\documentclass[12pt]{amsart}
\usepackage[applemac]{inputenc}
\usepackage{times, ulem, amsfonts}
\usepackage{mathpazo, amsmath, amssymb, amsthm,color}

\date{}
\definecolor{sah}{rgb}{0.66,0.33, 0.04}
\definecolor{adel4}{cmyk}{1,0,0,0}
\definecolor{adel3}{rgb}{0.66,0.33, 0.04}
\definecolor{adel1}{cmyk}{0,0.20,1,0}
\definecolor{adel2}{cmyk}{0,0.40,1,0.30}
\definecolor{adel0}{rgb}{0.99,0.60, 0.30}
\definecolor{trut}{rgb}{0.99,0.80, 0.00}
\definecolor{trus}{rgb}{0.00, 0.50, 0.00}
 \definecolor{trust}{rgb}{0.99, 0.99, 0.80}
\definecolor{MaCouleur}{rgb}{0,0.9,0.3}
\def\virgp{\raise 2pt\hbox{,}}
\def\({\left(}
\def\){\right)}

\def\<{\langle}
\def\>{\rangle}

\theoremstyle{plain}

\newtheorem{Theo}{Theorem}
 \newtheorem{Lemm}{Lemma}

\newtheorem{Prop}{Proposition}

 \newtheorem{Defin}{ Definition}
 
 \newtheorem{rema}{Remark}

\newcommand{\av}{{\rm Avg}}

\newcommand{\R}{{\mathbb R}}

\newcommand{\er}{{\mathbb R}}

\newcommand{\lb}{{{\rm LBMO}}}

\newcommand{\co}{{\rm o}}

\usepackage[textmath,displaymath,floats,graphics,delayed]{preview}
  \title[]{On the global well-posedness of the 2D Euler equations for a large class of Yudovich type data}

\author[F. Bernicot]{Fr\'ed\'eric Bernicot}
\address{CNRS - Universit\'e de Nantes \\ Laboratoire de Math\`ematiques Jean Leray \\ 2, Rue de la Houssini\`ere F-44322 Nantes Cedex 03, France}
 \email{frederic.bernicot@univ-nantes.fr}
\author[S. Keraani]{Sahbi Keraani}
\address{UFR de math\'ematiques \\ Universit\'e de Lille 1\\
 59655 Villeneuve d'Ascq  Cedex\\   France}
\email{sahbi.keraani@univ-lille1.fr}

\thanks{The two authors are supported by the ANR under the project AFoMEN no. 2011-JS01-001-01.}

\date{\today}

\subjclass[2000]{76B03 ; 35Q35}
\keywords{ 2D incompressible Euler equations, Global well-posedness, {\rm BMO}-type space}

\begin{document}

\begin{abstract}
The study of the 2D Euler equation with non Lipschitzian velocity was initiated by Yudovich in  \cite{Y1} where a result of global well-posedness for  essentially  bounded vorticity is proved.  A lot of  works have been since  dedicated to the extension of this result to more general spaces. To the best of our knowledge all these contributions lack the proof of   at least  one of the  following three fundamental properties:  global existence, uniqueness and regularity persistence. In this paper we introduce a Banach space containing  unbounded functions for which all these properties are shown to be satisfied. 
\end{abstract}

\maketitle

\section{Introduction}
We consider
the Euler system related to an incompressible inviscid fluid with constant
density, namely
\begin{equation}
\label{E}
 \left\{ 
\begin{array}{ll} 
\partial_t u+u\cdot\nabla u+\nabla P=0,\qquad x\in \mathbb R^d, t>0, \\
\nabla.u=0,\\
u_{\mid t=0}= u_{0}.
\end{array} \right.    
     \end{equation}
 Here, the vector field   \mbox{$u=(u_2,u_1,...,u_d)$}  is a function of  \mbox{$(t,x)\in \mathbb R_+\times\mathbb R^d$}  denoting the velocity of the fluid and  the scalar function  
\mbox{$P$}   stands for  the   pressure.
The second equation of the system  \mbox{$\nabla.u=0$}  is the 
  condition of incompressibility. 
Mathematically, it guarantees the preservation of  Lebesgue  measure by the particle-trajectory mapping (the classical flow associated to the 
velocity vector fields).
It is worthy of noting that the pressure can be  recovered from the velocity  via an explicit Calder\'on-Zygmund type operator (see \cite{Ch1}  for instance).

The question of  local  well-posedness  of \eqref{E} with smooth data was resolved by many authors in different spaces (see for instance \cite{Ch1,Maj}). In this context,  the vorticity  \mbox{$\omega={\rm curl}\,  u$}  plays a fundamental role. In fact, the well-known BKM criterion  \cite{Beale}  ensures that the development of finite time singularities for these solutions is related to the blow-up  of the  \mbox{$L^\infty$}  norm of the vorticity near the maximal time existence.   A direct consequence of this result is the global well-posedness of the two-dimensional Euler solutions with smooth initial data, since the vorticity  satisfies
the transport equation
\begin{equation}
\label{tourbillon}
\partial_t\omega+(u \cdot \nabla)\omega=0,
\end{equation}
and then all  its  \mbox{$L^p$}  norms are conserved.    

Another class of solutions requiring lower regularity on the velocity can be considered: the weak solutions (see for instance \cite[Chap 4]{lions1}). They solve a
weak form of the equation in the distribution sense, placing the equations in large
spaces and using duality. The  divergence form of Euler equations  allows to put all the derivative on  the test functions and so to obtain
$$
\int_0^\infty\int_{\R^d}(\partial_t\varphi+(u \cdot \nabla)\varphi).u\,dxdt+\int_{\R^d}\varphi(0,x)u_0(x)\,dx=0,
$$
for all  \mbox{$\varphi\in C^\infty_0(\R_+\times\R^d, \R^d)$}  with  \mbox{$  \nabla.\varphi=0$}. In the two dimensional space and when the regularity is sufficient to give a sense to Biot-Savart law, then one can consider an alternative weak formulation:  the vorticity-stream weak formulation. It  consists in resolving the weak form of \eqref{tourbillon} supplemented with the Biot-Savart law:
\begin{equation}
\label{bs}
u=K\ast\omega,\quad \hbox{with}\quad K(x)=\frac{x^\perp}{2\pi|x|^2}.
\end{equation}
In this case,   \mbox{$(v,\omega)$}  is a weak solution to the vorticity-stream formulation of the 2D Euler equation with initial data  \mbox{$\omega_0$}  if \eqref{bs} is satisfied and 
$$
\int_0^\infty\int_{\R^2}(\partial_t\varphi+u.\nabla\varphi)\omega(t,x) dxdt+\int_{\R^2}\varphi(0,x)\omega_0(x)dx=0,
$$
for all  \mbox{$\varphi\in C^\infty_0(\R_+\times\R^2,\R)$}.

The questions of existence/uniqueness of weak solutions  have been extensively studied and  a detailed
account can be found in the books \cite{Ch1, Maj, lions1}. We emphasize that, unlike the fixed-point argument, the compactness method does not guarantee the uniqueness of the  solutions and then  the two issues (existence/uniqueness) are usually  dealt with separately.  These questions have been originally addressed by Yudovich in \cite{Y1} where the existence and uniqueness of weak solution to 2D Euler systems (in bounded domain) are proved under the assumptions:   \mbox{$u_0\in L^2$}  and  \mbox{$\omega_0\in L^\infty$}.  
Serfati  \cite{Ser} proved the uniqueness and existence of a solution with initial velocity and vorticity which are only bounded (without any integrability condition). There is an  extensive literature on the existence of weak solution to Euler system, possibly without uniqueness, with unbounded vorticity. DiPerna-Majda \cite{DM} proved the existence of weak solution for  \mbox{$\omega_0\in L^1\cap L^p$}  with  \mbox{$2<p<\infty$}. The  \mbox{$L^1$}  assumption in  DiPerna-Majda's paper has been removed by Giga-Miyakawa-Osada \cite{GMO}.
Chae \cite{Ch} proved an existence result for  \mbox{$\omega_0$}  in  \mbox{$L\ln^+L$}  with compact support. 
More recently, Taniuchi  \cite{tan}  has proved the global existence  (possibly without uniqueness nor regularity persistence) for  \mbox{$(u_0,\omega_0)\in L^\infty\times {\rm BMO}$}. The papers \cite{Vishik1} and \cite{Y2} are concerned with the questions of existence  and uniqueness of weak solutions for larger classes of vorticity. Both have intersections with the present paper and we will come back to them at  the end of this section (Remark \ref{r2}). A framework for measure-valued
solutions can be found in  \cite{De} and \cite{FLX} (see also \cite{Ger} for more detailed references).

Roughly speaking, the proof of uniqueness of weak solutions requires a uniform, in time, bound of the   \mbox{$\log$}-Lipschitzian norm of the velocity. This ``almost" Lipschitzian regularity of the velocity is enough to assure the existence and uniqueness of the associated flow (and then of  the solution). Initial conditions of the type  \mbox{$\omega_0 \in L^\infty(\R^2)$}  ( or \mbox{$\omega_0 \in{\rm BMO}, B_{\infty,\infty}^0,...$})  guarantee the  \mbox{$\log$}-Lipschitzian regularity of  \mbox{$u_0$}. However, the persistence of  such regularity when time varies requires  an {\it a priori} bound of these quantities for the approximate-solution sequences. This is  trivially done (via the conservation law)  in the  \mbox{$L^\infty$}  case  but not at all clear for the other cases. The main issue in this context is the action of Lebesgue measure preserving homeomorphisms on these spaces. In fact, it is easy to prove that all these spaces are invariant under the action of such class of homeomorphisms, but the optimal form of the constants (depending on the homeomorphisms and important for the application) are not easy to find. It is worth of mentioning, in this context, that the  proof  by Vishik \cite{Vishik2} of the global existence for \eqref{E} in  the borderline Besov spaces  is based on a  refined result on  the action of Lebesgue  measure preserving homeomorphisms on   \mbox{$B_{\infty,1}^0$}. 

 In this paper we place ourselves in some Banach space which is strictly imbricated between  \mbox{$L^\infty$}  and  \mbox{${\rm BMO}$}. Although located beyond the reach of the conservation laws of the vorticity this space has many nice properties (namely with respect of the action of the group of Lebesgue measure preserving homeomorphisms) allowing to derive the above-mentioned {\it a priori} estimates for the approximate-solution sequences. 
 
Before going any further, let us introduce this functional space (details about  \mbox{${\rm BMO}$}  spaces can be found in the book of Grafakos \cite{GR}).
\begin{Defin} For a complex-valued locally integrable function on  \mbox{$\R^2$}, set
$$
\|f\|_{\lb}:=\|f\|_{{\rm BMO}}+\sup_{B_1,B_2}\frac{|\av_{B_2}(f)-\av_{B_1}(f)|}{1+\ln\big(\frac{ 1-\ln r_2  }{1-\ln r_1 }\big)},
$$
where the supremum is taken aver all pairs of balls  \mbox{$B_2=B(x_2,r_2)$}   and  \mbox{$B_1=B(x_1,r_1)$}  in  \mbox{$\R^2$}   with  \mbox{$0<r_1\leq 1$}  and  \mbox{$2B_2\subset B_1$}.
Here and subsequently, we denote 
$$
\av_{D}(g):=\frac1{|D|}\int_Dg(x)dx,
$$  
for every  \mbox{$g\in L^1_{\text{loc}}$}  and every non negligible set  \mbox{$D\subset \R^2$}. 
Also, for a ball  \mbox{$B$}  and  \mbox{$\lambda>0$},  \mbox{$\lambda B$}  denotes the ball that is concentric with  \mbox{$B$}  and whose radius is  \mbox{$\lambda$}  times the radius of  \mbox{$B$}.
\end{Defin}
We recall  that 
$$
\|f\|_{{\rm BMO}}:=\sup_{{\rm ball}\,\, B}\av_{B}|f-\av_{B}(f)|.
$$
It is worth of  noting that if  \mbox{$B_2$}   and  \mbox{$B_1$}  are  two balls such that  \mbox{$2B_2\subset B_1$}  then\footnote{  Throughout this paper  the notation   \mbox{$A \lesssim  B$}  means that there exists a positive  universal constant  \mbox{$C$}  such that  \mbox{$A\le CB$}. }
\begin{equation}
\label{22}
{|\av_{B_2}(f)-\av_{B_1}(f)|} \lesssim {\ln(1+\frac{r_1}{r_2})} \|f\|_{{\rm BMO}}.
\end{equation}
In the definition of  \mbox{$\lb$}  we replace the   term  \mbox{$\ln(1+\frac{r_1}{r_2})$}  by  \mbox{$\ln\big(\frac{ 1-\ln r_2  }{1-\ln r_1 }\big)$}, which is smaller.  This  puts more constraints on the functions belonging to this space\footnote{ Here, we identify all  functions whose difference is a constant. In section 2, we will prove that  \mbox{$\lb$}   is complete and strictly imbricated between  \mbox{${\rm BMO}$}  and  \mbox{$L^\infty$}.  {The "$L$" in  \mbox{$\lb$}  stands for "logarithmic".}
} and allows us to derive some crucial property on the composition of them with Lebesgue  measure preserving  homeomorphisms,  which is the heart of our analysis.

 The following statement  is the  main result of the paper.
\begin{Theo} 
\label{main}Assume  \mbox{$\omega_0\in L^p\cap \lb$}  with  \mbox{$p\in ]1,2[$}. Then there exists a unique global weak solution  \mbox{$(v,\omega)$}  to the vorticity-stream formulation of the 2D Euler  equation.  Besides, there exists a constant  \mbox{$C_0$}   depending only  on the   \mbox{$L^p\cap \lb$}-norm of  \mbox{$\omega_0$}  such that
\begin{equation}
\label{bound}
\|\omega(t)\|_{  L^p\cap \lb }\leq C_0\exp({C_0t}),\qquad\forall\,  t\in\R_{+}.
\end{equation}
\end{Theo}
Some remarks are in order.
\begin{rema} {\rm The proof  gives more, namely   \mbox{$
\omega\in \mathcal C(\R_+, L^q)$}  for all   \mbox{$p\leq q<\infty$}. 
Combined with the Biot-Savart law\footnote{If  \mbox{$\omega_0\in L^p$}  with  \mbox{$p\in ]1,2[$}  then a classical Hardy-Littlewood-Sobolev inequality gives   \mbox{$u\in L^q$}  with  \mbox{$\frac1q=\frac1p-\frac12$}. } this yields
\mbox{$
u\in \mathcal C(\R_+, W^{1,r})\cap \mathcal C(\R_+, L^\infty)$}  for all  \mbox{$\frac{2p}{2-p}\leq r<\infty$}.
}
\end{rema}

\begin{rema}
\label{r2}
{\rm The essential point of Theorem \ref{main} is that it  provides an initial  space which is strictly larger than  \mbox{$L^p\cap L^\infty$}  (it contains unbounded elements) which is a space of existence, uniqueness and persistence of regularity at once. We emphasize that the bound \eqref{bound} is crucial since it implies that  \mbox{$u$}  is, uniformly in time,  \mbox{$\log$}-Lipschitzian which is the main ingredient for the uniqueness. Once this bound established the uniqueness follows from the work by Vishik \cite{Vishik1}. In this paper Vishik also gave a result of existence (possibly without regularity persistence) in some large space characterized by  growth of the partial sum of the  \mbox{$L^\infty$}-norm of its dyadic blocs. 
 We should also mention the result of uniqueness by Yudovich \cite{Y2} which establish uniqueness (for bounded domain) for some space which contains unbounded functions. Note also that the example of unbounded function, given in \cite{Y2}, belongs actually to the space  \mbox{$\lb$}   (see Proposition \ref{pro3} below). Our approach is different from those in \cite{Vishik1} and \cite{Y2} and  uses a classical harmonic analysis ``\`a la stein" without making appeal to the Fourier analysis (para-differential calculus).  }
 \end{rema}
 \begin{rema} {\rm The main ingredient of the proof of \eqref{bound} is a logarithmic estimate in the space  \mbox{$L^p\cap \lb$}   (see Theorem  \ref{decom} below). It would be desirable to prove this result for  \mbox{${\rm BMO}$}  instead of  \mbox{$\lb$}. 
 Unfortunately, as it is proved in \cite{BK}, the  corresponding estimate with  \mbox{${\rm BMO}$}  is  optimal (with the bi-Lipschitzian norm instead of the  \mbox{$\log$}-Lipschitzian norm of the homeomorphism) and so the argument presented here seem to be not extendable to  \mbox{${\rm BMO}$}. }
\end{rema}

The remainder of this  paper is organized as follows. In the two next sections we introduce some functional spaces and prove a logarithmic estimate  which is crucial to the proof of Theorem \ref{main}.  The fourth and last  section is dedicated to the proof of Theorem \ref{main}.

     \section{Functional spaces}
    
   Let us first recall that the set of \mbox{$\log$}-Lipschitzian vector fields on $\R^2$ , denoted by $LL$, is the
set of bounded vector fields $v$ such that
    $$
   \|v\|_{LL}:=\sup_{x\neq y}\frac{|v(x)-v(y)|}{|x-y|\big(1+\big|\ln|x-y|\big|\big)}<\infty.
    $$
   The importance of this notion lies in the fact that if the vorticity belong to the Yudovich type space (say  \mbox{$L^1\cap L^\infty$}) then the velocity is no longer Lipschitzian, but  \mbox{$\log$}-Lipschitzian. In this case we still have existence and uniqueness of flow but a 
loss of regularity  may occur.  Actually, this   loss  of regularity is unavoidable and its degree is 
 related to the norm  \mbox{$L^1_t(LL)$}  of the velocity. The reader is referred to section 3.3 in  \cite{bah-ch-dan} for more details about this issue.
 
   To capture this behavior, and 
  overcome the difficulty generated by it, we introduce the following definition. 
  \begin{Defin} For every homeomorphism  \mbox{$\psi$}, we set
    $$
   \|\psi\|_*:=\sup_{x\neq y}\Phi\big(|\psi(x)-\psi(y)|, |x-y|\big),
    $$
   where  \mbox{$\Phi$}  is defined on  \mbox{$]0,+\infty[\times]0,+\infty[$}  by
   \begin{equation*}
\Phi(r,s)=\left\{ 
\begin{array}{ll} 
\max\{\frac{1+|\ln(s)| }{ 1+|\ln r | };\frac{ 1+|\ln r | }{1+|\ln(s)| }\},\quad {\rm if}\quad (1-s)(1-r)\geq 0, \\
{(1+|\ln s|) }{ (1+|\ln r|) },\quad {\rm if}\quad  (1-s)(1-r)\leq 0.
\end{array} \right.    
     \end{equation*}
     \end{Defin}
Since  \mbox{$\Phi$}  is symmetric then  \mbox{$\|\psi\|_*=\|\psi^{-1}\|_*\geq 1$}.  It is clear also that every homeomorphism  \mbox{$\psi$}  satisfying
$$
\frac{1}C|x-y|^\alpha\leq |\psi(x)-\psi(y)|\leq C|x-y|^\beta,
$$
for some  \mbox{$\alpha,\beta,C>0$}  has its  \mbox{$\|\psi\|_*$}   finite (see Proposition \ref{p1} for a reciprocal property). 

The definition above is motivated by this proposition (and by Theorem \ref{decom} below as well).
\begin{Prop}  \label{prop} Let  \mbox{$u$}  be a smooth divergence-free vector fields and  \mbox{$\psi$}  be its  flow:
$$
\partial_t{\psi}(t,x)=u(t,\psi(t,x)),\qquad {\psi}(0,x)=x.
$$
Then, for every  \mbox{$t\geq 0$}
$$
\|\psi(t,\cdot)\|_*\leq \exp(\int_0^t\|u(\tau)\|_{LL}d\tau).
$$
\end{Prop}
\begin{proof} It is well-known that for every  \mbox{$t\geq 0$}  the mapping  \mbox{$  x\mapsto \psi(t,x)$}  is a   Lebesgue  measure preserving homeomorphism  (see \cite{Ch1} for instance). We fix  \mbox{$t\geq 0$}  and  \mbox{$x\neq y$}  and  set 
$$
z(t)=|\psi(t,x)-\psi(t,y)|.
$$
Clearly the function  \mbox{$Z$}  is strictly positive and satisfies 
$$
|\dot{z}(t)|\leq  \|u(t)\|_{LL}(1+|\ln z(t)|)z(t).
$$  
Accordingly, we infer
$$
|g(z(t))-g(z(0))|\leq \int_0^t\|u(\tau)\|_{LL}d\tau
$$
where 
\begin{equation*}
g(\tau):=\left\{ 
\begin{array}{ll} 
\ln(1+\ln(\tau)),\quad {\rm if}\quad \tau\geq 1, \\
-\ln(1-\ln(\tau)),\quad {\rm if}\quad 0<\tau<1.
\end{array} \right.    
     \end{equation*}
     This yields in particular that
      \mbox{$
     \frac{\exp(g(z(t)))}{\exp(g(z(0)))}$}   and    \mbox{$\frac{\exp(g(z(0)))}{\exp(g(z(t)))}$}  are  both controlled by   \mbox{$\exp(\int_0^t\|u(\tau)\|_{LL}d\tau)$}  leading to 
  $$
 \Phi(z(t), z(0))\leq  \exp(\int_0^t\|u(\tau)\|_{LL}d\tau),
  $$
 as claimed.   
\end{proof}
The following proposition follows directly from the definition by a  straightforward computation.
\begin{Prop} 
\label{p1}
Let  \mbox{$\psi$}  be a homeomorphism with  \mbox{$\|\psi\|_{*}<\infty$}. Then for every  \mbox{$(x,y)\in \mathbb R^2\times\mathbb R^2$}  one has
\begin{enumerate}
\item If  \mbox{$|x-y|\geq 1$}  and  \mbox{$|\psi(x)-\psi(y)|\geq 1$}
$$
e^{-1}|x-y|^{\frac{1}{\|\psi\|_*}}\leq |\psi(x)-\psi(y)|\leq  e^{\|\psi\|_*}|x-y|^{\|\psi\|_*}.
$$
\item  If  \mbox{$|x-y|\leq 1$}  and  \mbox{$|\psi(x)-\psi(y)|\leq 1$}
$$
e^{-\|\psi\|_*}|x-y|^{{\|\psi\|_*}}\leq |\psi(x)-\psi(y)|\leq e |x-y|^{\frac{1}{\|\psi\|_*}}.
$$
\item  In the other cases 
$$
e^{-\|\psi\|_*}|x-y|\leq |\psi(x)-\psi(y)|\leq  e^{\|\psi\|_*}|x-y|.
$$
\end{enumerate}
\end{Prop}
As an application we obtain the following useful lemma.
\begin{Lemm}
\label{g}
 For every  \mbox{$r>0$}  and a homeomorphism  \mbox{$\psi$}   one has
$$4\psi(B(x_0,r))\subset B(\psi(x_0), g_\psi(r)),
$$  
where\footnote{This notation means that for every ball  \mbox{$B\subset\psi(B(x_0,r))$}  we have  \mbox{$4B \subset B(\psi(x_0), g_\psi(r))$}.},
\begin{equation*}
g_\psi(r):=\left\{ 
\begin{array}{ll} 4e^{ \|\psi\|_{*}}r^{ \|\psi\|_{*}},\quad {\rm if}\quad r\geq 1, \\
4\max\{ er^{\frac{1}{\|\psi\|_{*}}}; e^{\|\psi\|_{*}}r\}, \quad {\rm if}\quad  0<r<1.
\end{array} \right.    
     \end{equation*}
In particular,
\begin{equation}
\label{ss}
 |\ln\Big(\frac{ 1+|\ln g_\psi(r)|  }{1+|\ln r|} \Big)|\lesssim  1+\ln\big(1+\|\psi\|_{*}\big).
\end{equation}
\end{Lemm}\begin{proof}
The first inclusion follows from Proposition \ref{p1} and the definition of  \mbox{$g_\psi$}. Let us check (\ref{ss}).
This comes from an easy computation using the following trivial fact: if  \mbox{$\alpha,\beta,\gamma>0$}  then 
$$
\sup(\beta, \frac1\beta)\leq \alpha^\gamma \Longleftrightarrow   |\ln(\beta)|\leq\gamma \ln(\alpha).
$$
\mbox{$\bullet$}  If   \mbox{$r\geq 1$}  then 
$$
1\leq \frac{ 1+|\ln g_\psi(r)|  }{1+|\ln r|}=\frac{ 1+\ln4+\|\psi\|_{*}+\ln r  }{1+\ln r}\leq 3+\|\psi\|_{*}.
$$
\mbox{$\bullet$}  If  \mbox{$r< 1$}  then we have to deal with two possible values of  \mbox{$g_\psi(r)$}.

\underline{\it Case 1:} If   \mbox{$g_\psi(r)=4er^{\frac{1}{\|\psi\|_{*}}}$}  then
$$
|\ln g_\psi(r)| =|\ln 4+1+\|\psi\|_{*}^{-1}\ln(r)|.
$$
Since  \mbox{$\|\psi\|_{*}\geq 1$}, we get
$$
\frac{ 1+|\ln g_\psi(r)|  }{1+|\ln r|}\leq \frac{ 3+\frac1{ \|\psi\|_{*} }|\ln r|  }{1+|\ln r|}\leq \frac{ 3+|\ln r|  }{1+|\ln r|}\leq 3.
$$
To estimate  \mbox{$\frac{1+|\ln r|}{ 1+|\ln g_\psi(r)|  }$}  we consider two possibilities.

- If  \mbox{$|\ln(r)|\leq 8 \|\psi\|_{*}$}  then 
$$
\frac{1+|\ln r|}{ 1+|\ln g_\psi(r)|  }\leq 1+|\ln r|\leq 1+8 \|\psi\|_{*}.
$$
- If  \mbox{$|\ln(r)|\geq 8 \|\psi\|_{*}$}  then 
$$ 
|\ln(4)+1+\|\psi\|_{*}^{-1}\ln(r)|\geq \frac12 \|\psi\|_{*}^{-1} |\ln(r)|,
$$
and so 
$$
\frac{1+|\ln r|}{ 1+|\ln g_\psi(r)|  }\leq \frac{1+|\ln r|}{1+\frac12 \|\psi\|_{*}^{-1} |\ln(r)|}\leq 2(1+\|\psi\|_{*}).
$$
\underline{\it Case 2:} If   \mbox{$g_\psi(r)=4e^{\|\psi\|_{*}}r$}  then
$$
|\ln g_\psi(r)| =|\ln 4+\|\psi\|_{*}+\ln(r)|.
$$
Thus,
$$
\frac{ 1+|\ln g_\psi(r)|  }{1+|\ln r|}\leq \frac{ 3+\|\psi\|_{*} +|\ln r|  }{1+|\ln r|}\leq 3+\|\psi\|_{*}.
$$
As previously for estimating   \mbox{$\frac{1+|\ln r|}{ 1+|\ln g_\psi(r)|  }$}, we consider two possibilities.

- If  \mbox{$|\ln(r)|\leq 2(\ln(4)+ \|\psi\|_{*})$}  then 
$$
\frac{1+|\ln r|}{ 1+|\ln g_\psi(r)|  }\leq 1+|\ln r|\leq 5+2 \|\psi\|_{*}.
$$
- If  \mbox{$|\ln(r)|\geq 2(\ln 4+ \|\psi\|_{*})$}  then  \mbox{$|\ln(4)+\|\psi\|_{*}+\ln r|\geq \frac12|\ln(r)|$}
and so 
$$
\frac{1+|\ln r|}{ 1+|\ln g_\psi(r)|  } \leq \frac{1+|\ln r|}{1+\frac12|\ln(r)|}\leq 2.
$$
\end{proof}

\begin{rema}
\label{sss} The estimate \eqref{ss} remains  valid when we multiply  \mbox{$g_\psi(r)$}  by any positive constant. 
\end{rema}

\section{ The  \mbox{$\lb$}  space}
Let us now detail some properties of the space  \mbox{$\lb$}  introduced in the first section of this paper.
\begin{Prop} 
\label{pro3}
The following properties hold true.\\
{\rm (i)} The space  \mbox{$\lb$}  is a Banach space included in  \mbox{${\rm BMO}$}  and strictly containing  \mbox{$L^\infty(\R^2)$}.\\
{\rm (ii)} For every  \mbox{$g\in \mathcal C^\infty_0(\R^2)$}  and  \mbox{$f\in \lb$}  one has
\begin{equation}
\| g\ast f\|_{\lb}\leq \|g\|_{L^1}\|  f\|_{\lb}.
\label{eq:comp} \end{equation}
\end{Prop}
\begin{proof}
(i) Completeness of the space. Let  \mbox{$(f_n)_n$}  be a Cauchy sequence in  \mbox{$\lb$}. Since  \mbox{${\rm BMO}$}  is complete then this sequences converges in  \mbox{${\rm BMO}$}  and then in  \mbox{$L^1_{\text{loc}}$}. 
Using the definition and the the convergence in  \mbox{$L^1_{\text{loc}}$}, we get that the convergence holds in  \mbox{$\lb$}.

 It remains to check that  \mbox{$L^\infty \subsetneq \lb$}. Since  \mbox{$L^\infty$}  is obviously embedded into  \mbox{$\lb$}, we have just to build an unbounded function belonging to  \mbox{$\lb$}. Take 
\begin{equation*}
f(x)=\left\{ 
\begin{array}{ll} \ln(1-\ln|x|) \qquad {\rm if}\quad |x|\leq 1\\
0,\qquad \qquad {\rm if}\quad  |x|\geq 1.
\end{array} \right.    
     \end{equation*}
     
It is clear that both  \mbox{$f$}  and  \mbox{$\nabla f$}  belong to  \mbox{$L^2(\R^2)$}   meaning that   \mbox{$f\in H^1(\R^2)\subset  {\rm BMO}$}.

Before going further three preliminary remarks are necessary. 

\mbox{$\bullet$}  Since  \mbox{$f$}  is radially  symmetric and decreasing  then, for every  \mbox{$r>0$},  the mapping   \mbox{$x\mapsto \av_{B(x,r)}f$}   is radial and decreasing.

\mbox{$\bullet$}  For the same reasons the mapping    \mbox{$r\mapsto \av_{B(0,r)}(f)$}   is decreasing.

\mbox{$\bullet$}  Take \mbox{$(r,\rho) \in ]0,+\infty[^2$}  and consider the problem of maximization of
   \mbox{$\av_{B(x_1,r)}(f)-\av_{B(x_2,r)}(f)$}  when   \mbox{$|x_1-x_2|=\rho$}. The convexity of  \mbox{$f$}  implies that 
\mbox{$x_1=0$}  and  \mbox{$|x_2|=\rho$}  are solutions of this problem.

\

We fix  \mbox{$r_1$}  and  \mbox{$r_2$}  such that  \mbox{$r_1\leq 1$}  and  \mbox{$2r_2\leq r_1$}.
For every  \mbox{$x_1\in\R^2$}  one defines  \mbox{$\tilde x_1$}  and  \mbox{$\hat x_1$}  as follows:
\begin{equation*}
\tilde x_1=\left\{ 
\begin{array}{ll} x_1(1-\frac{r_2+r_1}{|x_1|}) \qquad {\rm if}\quad |x_1|\geq r_2+r_1\\
0,\qquad \qquad {\rm if}\quad  |x_1|\leq r_2+r_1,
\end{array} \right.    
     \end{equation*}
and
\begin{equation*}
\hat x_1=\left\{ 
\begin{array}{ll} x_1(1+\frac{r_2+r_1}{|x_1|}) \qquad {\rm if}\quad |x_1|\neq 0\\
({r_2+r_1},0)\qquad \qquad {\rm if}\quad  |x_1|=0.
\end{array} \right.    
     \end{equation*}
Let
\mbox{$A(x_1)$} be the set of  admissible  \mbox{$x_2$}:  the set of  \mbox{$x_2$}  such that  \mbox{$2B(x_2,r_2)\subset B(x_1,r_1)$}. 
Using the two preliminary remarks above, we see that 
$$
\sup_{ x_2\in A(x_1)}|\av_{B(x_2,r_2)}(f)-\av_{B(x_1,r_1)}(f)|\leq \max\{J_{1},J_{2}\}.
$$
with
\begin{eqnarray*}
J_{1}&=&\av_{B(\tilde x_1,r_2)}(f)-\av_{B(x_1,r_1)}(f),
\\
J_{2}&=&\av_{B(x_1,r_1)}(f)-\av_{B(\hat{x}_1,r_2)}(f).
\end{eqnarray*} 
In fact, if  \mbox{$\av_{B(x_2,r_2)}(f)-\av_{B(x_1,r_1)}(f)$}  is positive (resp. negative) then it is obviously  dominated by  \mbox{$J_{1}$}  (resp.  \mbox{$J_{2}$}).
Thus, we obtain
$$
\sup_{ x_2\in A(x_1)}|\av_{B(x_2,r_2)}(f)-\av_{B(x_1,r_1)}(f)|\leq J_{1}+J_{2}= \av_{B(\tilde x_1,r_2)}(f)-\av_{B(\hat{x}_1,r_2)}(f).
$$
The right hand side is maximal in the configuration when    \mbox{$\tilde x_1=0$}  and  \mbox{$\hat{x}_1$}  the furthest away from  \mbox{$0$}.
This means when 
\mbox{$|x_1|=r_1+r_2$},   \mbox{$\tilde x_1=0$}  and   \mbox{$|\hat{x}_1|=2(r_1+r_2)$}.

Since   \mbox{$f$}  is increasing (going to the axe) then
$$
 \av_{B(\hat{x}_2,r_1)}(f)\geq f(4r_1).
  $$
 Finally, we get for all  \mbox{$x_1\in\mathbb R^2$}  and  \mbox{$ x_2\in A(x_1)$}
 \begin{eqnarray*}
|\av_{B(x_2,r_2)}(f)-\av_{B(x_1,r_1)}(f)|\leq \av_{B(0,r_2)}(f)- f(4r_1).
\end{eqnarray*}
Now it is easy to see that
$$
f(4r_1)= \ln(1-\ln(r_1))+ {\mathcal O}(1),
$$
and  (with an integration by parts)
\begin{eqnarray*}
\av_{B(0,r_2)}(f) 
& =& \ln(1-\ln(r_2)) + \frac{1}{r_2^2}\int_0^{r_1} \frac{1}{1-\ln(\rho)} \rho d\rho 
\\
& =& \ln(1-\ln(r_2)) + {\mathcal O}(1).
\end{eqnarray*}
This yields,
$$
|\av_{B(x_2,r_2)}(f)-\av_{B(x_1,r_1)}(f)|\leq \ln\Big(\frac{1-\ln(r_2)}{1-\ln(r_1)}\Big)+ {\mathcal O}(1),
$$
as desired.

\

(ii) Stability by convolution. (\ref{eq:comp}) follows from the fact that for all  \mbox{$r>0$}  
$$
x\mapsto\av_{B(x,r)}(g\ast f)=(g\ast\av_{B(\cdot,r)}(f))(x).
$$

\end{proof}
The advantage of using the space  \mbox{$\lb$}  lies in the following logarithmic estimate which is the main ingredient for proving Theorem \ref{main}.
\begin{Theo}
\label{decom}
There exists a universal constant  \mbox{$C>0$}  such that 
$$
\|f{\rm o}\psi\|_{\lb\cap L^p}\leq C\ln(1+\|\psi\|_*)\|f\|_{\lb\cap L^p},
$$
for any  Lebesgue  measure preserving  homeomorphism  \mbox{$\psi$}.
\end{Theo}
\begin{proof}[Proof of Theorem \ref{decom}] 
Of course we are concerned with $\psi$ such that $\|\psi\|_*$ is finite (if not the inequality is empty).
 Without loss of generality one can assume that  \mbox{$\|f\|_{\lb\cap L^p}=1$}.  Since  \mbox{$\psi$}  preserves Lebesgue  measure then the  \mbox{$L^p$}-part of the norm is conserved. For the  two other parts of the norm, we will proceed in two steps. In the first step we consider the  \mbox{${\rm BMO}$} term of the norm and in the second one  we deal with the other term. 
\subsection*{ Step 1} Let   \mbox{$B=B(x_0,r)$} be  a given ball of  \mbox{$\R^2$}. 
By using the  \mbox{$L^p$}-norm we need only to deal with balls whose radius is smaller than a universal constant  \mbox{$\delta_0$}  (we want   \mbox{$r$}  to be small with respect to the constants appearing in Whitney covering lemma below).  Since  \mbox{$\psi$}  is a Lebesgue  measure preserving homeomorphism   then  \mbox{$\psi(B)$}  is an open connected\footnote{We have also that  \mbox{$  \psi(B)^C=\psi(B^C)$}  and   \mbox{$\psi(\partial B)=\partial(\psi(B)).$}    } set with  \mbox{$|\psi(B)|=|B|$}. By Whitney covering lemma, there exists a collection of balls  \mbox{$(O_j)_j$}  such that: 

- The collection of double ball is a bounded covering:
$$
\psi(B)\subset \bigcup 2O_j.
$$

- The collection is disjoint and, for all  \mbox{$j$}, 
$$
O_j\subset \psi(B).
$$

- The Whitney property is verified:
$$
r_{O_j}\simeq d(O_j, \psi(B)^c).
$$

\

\mbox{$\bullet$}   {\it Case 1}:  \mbox{$r\leq\frac14 e^{-\|\psi\|_*}$}. In this case
$$
g_\psi(r)\leq 1.
$$
We set  \mbox{$\tilde B:= B(\psi(x_0), g_\psi(r))$}. 
Since  \mbox{$\psi$}  preserves  Lebesgue  measure we get

\begin{eqnarray*}
 \av_{B}|f\co\psi- \av_{B}(f\co\psi)|&=&\av_{\psi(B)}|f- \av_{\psi(B)}(f)|
\\
&\leq &    2  \av_{\psi(B)}|f- \av_{\tilde B}(f)|.
\end{eqnarray*}
Using the notations above 
\begin{eqnarray*}
\av_{\psi(B)}|f- \av_{\tilde B}(f)|&\lesssim & \frac{1}{|B|}\sum_j |O_j|\av_{2O_j}\big|f- \av_{\tilde B}(f)\big|
\\
&\lesssim & I_1+I_2,
\end{eqnarray*}
with
\begin{eqnarray*}
I_1&=& \frac{1}{|B|}\sum_j |O_j|\big |\av_{2O_j}(f)- \av_{2O_j}(f)\big |\\
I_2&= & \frac{1}{|B|}\sum_j |O_j|\big |\av_{2O_j}(f)- \av_{\tilde B}(f)\big |.
\end{eqnarray*}
On one hand, since   \mbox{$\sum|O_j|\leq |B|$}  then
\begin{eqnarray*}
I_1&\leq& \frac{1}{|B|}\sum_j |O_j|\|f\|_{{\rm BMO}}
\\
&\leq & \|f\|_{{\rm BMO}}.
\end{eqnarray*}
On the other hand, sinc \mbox{$4O_j\subset  \tilde B$}  (remember Lemma \ref{g}) and   \mbox{$r_{\tilde B}\leq 1$}, it ensues that
\begin{eqnarray*}
I_2&\lesssim&   \frac{1}{|B|}\sum_j |O_j|\big(1+\ln\Big(\frac{1-\ln 2r_j}{ 1-\ln g_\psi(r)  } \Big)\big)
\\
&\lesssim&  \frac{1}{|B|}\sum_j |O_j|(1+\ln\big(\frac{1-\ln r_j}{ 1-\ln g_\psi(r)  } \big)).
\end{eqnarray*}
Thanks to  \eqref{ss}  we get
\begin{eqnarray}
\nonumber
\ln\Big(\frac{1-\ln r_j}{ 1-\ln g_\psi(r)  } \Big)&\leq& \ln\Big(\frac{1-\ln r_j}{ 1-\ln r  } \Big)+\ln\Big(\frac{1-\ln r}{ 1-\ln g_\psi(r)  } \Big)
\\
\label{s}
&\lesssim& 1+  \ln\Big(\frac{1-\ln r_j}{ 1-\ln r  } \Big)+\ln(1+\|\psi\|_{*}).
\end{eqnarray}
Thus it remains to prove that
\begin{eqnarray}
\label{ef}
 II:=\frac{1}{|B|}\sum_j |O_j|(1+\ln\big(\frac{1-\ln r_j} { 1-\ln r  }\big))\lesssim 1+\ln(1+\|\psi\|_{*}).
\end{eqnarray}
For every  \mbox{$k\in\mathbb N$}  we set 
$$
u_k:=\sum_{e^{-(k+1)}r< r_j\leq e^{-k}r}|O_j|,
$$
so that
\begin{eqnarray}
\label{eff}
II\leq \frac{1}{|B|}\sum_{k\geq 0} u_k\big(1+\ln\big(\frac{k+2-\ln r} { 1-\ln r  }\big)\big).
\end{eqnarray}

We need the following lemma.
\begin{Lemm} 
\label{equivalence}
There exists a universal constant  \mbox{$C>0$}  such that
$$
u_k\leq Ce^{-\frac{k}{\|\psi\|_*}}r^{1+\frac{1}{\|\psi\|_*}},
$$
for every  \mbox{$k\in\mathbb N$}.
\end{Lemm}
\begin{proof}[Proof of Lemma \ref{equivalence}]
If we denote by   \mbox{$C\geq 1$}  the implicit constant appearing in Whitney Lemma, then 
$$
u_k\leq |\{ y\in \psi(B): d(y, \psi(B)^c)\leq Ce^{-k}r\}|.
$$
The preservation of  Lebesgue  measure by  \mbox{$\psi$}  yields
$$
 |\{ y\in \psi(B): d(y, \psi(B)^c)\leq Ce^{-k}r\}|=|\{ x\in B: d(\psi(x), \psi(B)^c)\leq Ce^{-k}r\}|,
$$

Since  \mbox{$  \psi(B)^c=\psi(B^c)$}  then
$$
u_k\leq |\{ x\in B: d(\psi(x), \psi(B^c))\leq Ce^{-k}r\}|.
$$
We set
  $$D_k=\{ x\in B: d(\psi(x), \psi(B^c))\leq Ce^{-k}r\}.
  $$
Since  \mbox{$\psi(\partial B)$}  is the frontier of  \mbox{$\psi(B)$}  and  \mbox{$d(\psi(x), \psi(B^c))=d(\psi(x), \partial \psi(B))$}  then
$$
D_k\subset \{ x\in B: \exists y\in \partial B \;{\rm with}\; |\psi(x)- \psi(y)|\leq Ce^{-k}r\}.
$$
The condition on  \mbox{$\delta_0$}  is just to assure that  \mbox{$Cr\leq 1$}  for all  \mbox{$r\leq\delta_0$}.
In this case  Proposition \ref{p1} gives
$$
D_k\subset \{ x\in B: \exists y\in \partial B: |x- y|\leq Ce^{1-\frac{k}{\|\psi\|_*}}r^{\frac{1}{\|\psi\|_*}}\}.
$$
Thus,  \mbox{$D_k$}  is contained in the annulus  \mbox{$\mathcal A=\{ x\in B: d(x,\partial B) \leq Ce^{1-\frac{k}{\|\psi\|_*}}r^{\frac{1}{\|\psi\|_*}}\}$}  and so 
$$
u_k\leq |D_k|\lesssim e^{-\frac{k}{\|\psi\|_*}}r^{1+\frac{1}{\|\psi\|_*}},
$$
as claimed.
\end{proof}

Coming back to \eqref{eff}. Let  \mbox{$N$}  a large integer to be chosen later. We split the sum in the right hand side of \eqref{eff} into two parts
$$
II\lesssim \sum_{k\leq N}(...)+\sum_{k> N}(.....):=II_{1}+II_{2}.
$$
Since  \mbox{$\sum u_k\leq |B|$}  then
\begin{eqnarray}
\label{ff}
II_{1}\leq 1+\ln\Big(\frac{N+2-\ln r} { 1-\ln r  }\Big).
\end{eqnarray}
On the other hand
$$
II_{2}\leq \sum_{k> N}e^{-\frac{k}{\|\psi\|_*}}r^{\frac{1}{\|\psi\|_*}-1}(1+\ln\big(\frac{k+2-\ln r} { 1-\ln r  }\big)).
$$
The parameter  \mbox{$N$}  will be taken bigger than  \mbox{$\|\psi\|_*$}  so that the function in  \mbox{$k$}  inside the sum is decreasing and an easy comparison with integral yields
\begin{eqnarray}
\label{fff}
 II_{2}\lesssim e^{-\frac{N}{\|\psi\|_*}}\|\psi\|_*^2r^{\frac{1}{\|\psi\|_*}-1}\big(1+\ln\Big(\frac{N+2-\ln r} { 1-\ln r  }\Big)\big).
\end{eqnarray}

Putting \eqref{ff} and \eqref{fff} together and taking  \mbox{$N= [\|\psi\|_*(\|\psi\|_*-\ln r)]+1$}  
\begin{eqnarray*}
II\lesssim \big(1+ e^{-\frac{N}{\|\psi\|_*}}\|\psi\|_*^2r^{\frac{1}{\|\psi\|_*}-1}\big)\big(1+\ln\Big(\frac{N+2-\ln r} { 1-\ln r  }\Big)\big).
\end{eqnarray*}
Taking  \mbox{$N= [\|\psi\|_*(\|\psi\|_*-\ln r)]+1$}  
\begin{eqnarray*}
II&\lesssim& (1+ e^{-\|\psi\|_*}\|\psi\|_*^2r^{\frac{1}{\|\psi\|_*}})\big(1+\ln\big(\frac{\|\psi\|_*(\|\psi\|_*-\ln r)+2-\ln r} { 1-\ln r  }\big)\big).
\\
&\lesssim& 1+\ln(1+\|\psi\|_*),
\end{eqnarray*}
where  we have used the fact that  \mbox{$r\leq 1$}  and the obvious inequality 
$$
\frac{\|\psi\|_*(\|\psi\|_*-\ln r)+2-\ln r} { 1-\ln r  }\lesssim (1+\|\psi\|_*)^2.
$$
This ends the proof of \eqref{ef}.

  \mbox{$\bullet$}  {\it Case 2:}  \mbox{$\delta_0\geq r \geq \frac14 e^{-\|\psi\|_*}$.}    In this case 
  $$
 |\ln r|\lesssim \|\psi\|_*.
  $$
 Since  \mbox{$\psi$}  preserves  Lebesgue  measure, we get
\begin{eqnarray*}
I&:=&\av_{B}|f\co\psi- \av_{B}(f\co\psi)|
\\
&\leq &   2   \av_{\psi(B)}|f|.
\end{eqnarray*}
Let \mbox{$\tilde O_j$}  denote the ball which is concentric to  \mbox{$O_j$}  and  whose radius is equal to  \mbox{$1$} (we use the same Whitney covering as above).  Without loss of generality  we can  assume  \mbox{$\delta_0\leq\frac14$}. This guarantees \mbox{$4O_j\subset\tilde O_j$} and yields by definition
\begin{eqnarray*}
I&\lesssim & \frac{1}{|B|}\sum_j |O_j|\av_{2O_j}|f-\av_{\tilde O_j}(f)|+ \frac{1}{|B|}\sum_j |O_j| |\av_{\tilde O_j}(f)|
\\
&\lesssim&   \frac{1}{|B|}\sum_j |O_j|\Big(1+\ln\big({1-\ln 2r_j} \big)\Big)\|f\|_{\lb}+\frac{1}{|B|}\sum_j |O_j| \|f\|_{L^p}
\\
&\lesssim& 
 1+  \frac{1}{|B|}\sum_j |O_j|\big(1+\ln\big({1-\ln r_j} \big)\big).
\end{eqnarray*}
As before one writes
\begin{eqnarray*}
I&\lesssim& \frac{1}{|B|}\sum_{k\geq 0} u_k\big(1+\ln\big(k+2-\ln r\big)\big)
\\
&\lesssim&1+\ln\big(N+2-\ln r\big)+ e^{-\frac{N}{\|\psi\|_*}} \|\psi\|_*^2 r^{\frac{1}{\|\psi\|_*}-1}\big(1+\ln\big(N+2-\ln r)\big).
\end{eqnarray*}
Taking  \mbox{$N=[ \|\psi\|_*(\|\psi\|_*-\ln r)]+1$}  and using the fact that  \mbox{$|\ln r|\lesssim \|\psi\|_*$}  leads to the desired result.

The outcome of this first step of the proof is
$$
\|f{\rm o}\psi\|_{{\rm BMO}\cap L^p}\lesssim\ln(1+\|\psi\|_*)\|f\|_{\lb\cap L^p}.
$$
\subsection*{ Step 2} This step of the proof deals with the second term in the  \mbox{$\lb$}-norm. It is shorter  than the first step  because it makes use of  the arguments   developed  above.
Take   \mbox{$B_2=B(x_2,r_2)$}   and  \mbox{$B_1=B(x_1,r_1)$}  in  \mbox{$\R^2$}   with   \mbox{$r_1\leq 1$}  and  \mbox{$2B_2\subset B_1$}.
There are three cases to consider.

\mbox{$\bullet$}  {\it Case 1:}  \mbox{$   r_1\lesssim e^{-\|\psi\|_*}$}  (so that  \mbox{$g_\psi(r_2)\leq g_\psi(r_1) \leq \frac12$}).

We set  \mbox{$\tilde B_i:= B(\psi(x_i), g_\psi(r_i)), i=1,2$}  and
$$
J:=\frac{|\av_{B_2}(f\co\psi)-\av_{B_1}(f\co\psi)|}{1+ \ln\big(\frac{ 1-\ln r_2 }{1-\ln r_1}\big)}.
$$
Since the denominator is bigger than  \mbox{$1$}  one get
$$
J\leq J_{1}+J_{2}+J_3,
$$
with 
\begin{eqnarray*}
J_{1}&=&  |\av_{\psi(B_2)}(f)-\av_{\tilde B_2}(f)|+ |\av_{\psi(B_1)}(f)-\av_{\tilde B_1}(f)| \\
J_{2}&=&\frac{|\av_{\tilde B_2}(f)-\av_{2\tilde B_1}(f)|}{1+ \ln\Big(\frac{ 1-\ln r_2 }{1-\ln r_1}\Big)}
 \\
J_{3}&=&|\av_{\tilde B_1}(f)-\av_{2\tilde B_1}(f)|.
\end{eqnarray*}
Since  \mbox{$2\tilde B_2\subset 2\tilde B_1$}   and  \mbox{$r_{2\tilde B_1}\leq1$}  then
$$
J_{2}\leq \frac{1+ \ln\big(\frac{ 1-\ln g_\psi(r_2) }{1-\ln(2g_\psi(r_1))}\big)}{1+ \ln\big(\frac{ 1-\ln r_2 }{1-\ln r_1}\big)}\|f\|_{\lb}.
$$
Using similar argument than \eqref{s} (and remembering Remark \ref{sss}) we infer
\begin{eqnarray*}
 \ln\Big(\frac{ 1-\ln g_\psi(r_2)  }{1-\ln (2g_\psi(r_1)) }\Big)
\lesssim 1+\ln(1+\|\psi\|_{*})+\ln\Big(\frac{ 1-\ln r_2  }{1-\ln r_1}\Big).
\end{eqnarray*}
Thus,
$$
J_{2}\lesssim1+\ln(1+\|\psi\|_{*}).
$$
The estimation \eqref{22} yields
$$
J_3\lesssim \|f\|_{{\rm BMO}}.
$$  
The term  \mbox{$J_{1}$}  can be handled exactly as in the analysis of   \mbox{case 1} of   \mbox{step 1}. 

\

\mbox{$\bullet$}  {\it Case 2:} \mbox{$e^{-\|\psi\|_*}\lesssim r_2$}. In this case we write
$$
J\leq \av_{\psi(B_2)}|f|+\av_{\psi(B_1)}|f|.
$$
Both terms can be handled as in the analysis of \mbox{case 2} of the proof of  \mbox{${\rm BMO}$}-part in   \mbox{step 1.} 

\mbox{$\bullet$}  {\it Case 3:}  \mbox{$r_2\lesssim e^{-\|\psi\|_*}$}  and  \mbox{$r_1\gtrsim e^{-\|\psi\|_*}  $}.  Again since the denominator is bigger than  \mbox{$1$}  we get
$$
J\leq \av_{\psi(B_2)}|f- \av_{\tilde B_2}(f)   |+\frac{|\av_{\tilde B_2}(f)|}{1+\ln\big(\frac{ 1-\ln r_2 }{1-\ln r_1}\big)}  +\av_{\psi(B_1)}|f|=J_{1}+J_{2}+J_3.
$$
The terms  \mbox{$J_{1}$}  and  \mbox{$J_3$}  can be controlled as before. The second term is controlled as follows (we make appear the average on  \mbox{$B(\psi(x_2),1)$}  and use Lemma \ref{g} with  \mbox{$\|f\|_{L^p}\leq 1$})
\begin{eqnarray*}
J_{2}&\leq& \frac{1+ \ln(1-\ln r_2)  }{1+ \ln\Big(\frac{ 1-\ln r_2 }{1-\ln r_1} \Big)} 
\\
&\leq& {1+ \ln(1+|\ln r_1|)  }
 \\
&\leq& {1+ \ln(1+\|\psi\|_*)  }.
\end{eqnarray*}
\end{proof}

\section{Proof of Theorem \ref{main}}
The proof falls naturally into three parts.
\subsection{{ A priori} estimates}
The following  estimates follow directly from Proposition \ref{prop} and Theorem \ref{decom}.
 \begin{Prop}
 \label{apriori} Let  \mbox{$u$}  be a smooth solution of \eqref{E} and   \mbox{$\omega$}  its vorticity. Then, there exists a constant  \mbox{$C_0$}  depending only on the norm  \mbox{$L^p\cap \lb$}  of  \mbox{$\omega_0$}  such that
   $$
  \|u(t)\|_{LL}+\|\omega(t)\|_{\lb}\leq C_0\exp{(C_0t)},
   $$
  for every  \mbox{$t\geq 0$}.
 \end{Prop}
 \begin{proof} One has  \mbox{$\omega(t,x) =\omega_0(\psi_t^{-1}(x))$}  where  \mbox{$\psi_t$}  is the flow associated to the velocity  \mbox{$u$}.  Since    \mbox{$u$}  is smooth then   \mbox{$\psi_t^{\pm 1}$}  is Lipschitzian for every  \mbox{$t\geq 0$}. This implies in particular that
  \mbox{$\|\psi_t^{\pm 1}\|_*$}  is finite for every  \mbox{$t\geq 0$}. 
 Theorem \ref{decom} and Proposition \ref{prop} yield together
\begin{eqnarray*}
 \|\omega(t)\|_{\lb}&\leq& C\|\omega_0\|_{\lb\cap L^p}\ln(1+\|\psi_t^{-1}\|_*)
 \\
 &\leq& C\|\omega_0\|_{\lb\cap L^p}\ln(1+\exp(\int_0^t\|u(\tau)\|_{LL}d\tau))
 \\
 &\leq& C_0(1+\int_0^t\|u(\tau)\|_{LL}d\tau).
  \end{eqnarray*}
On the other hand, one has
  \begin{eqnarray*}
 \|u(t)\|_{LL}&\leq&\|\omega(t)\|_{L^2}+ \|\omega(t)\|_{B_{\infty,\infty}^0}
 \\
 &\leq& C(\|\omega_0\|_{L^2}+ \|\omega(t)\|_{BM0}).
\end{eqnarray*}
The first estimate is classical (see \cite{bah-ch-dan} for instance) and the second one is just the conservation of the  \mbox{$L^2$}-norm of the vorticity and the continuity of the embedding  \mbox{${\rm BMO}\hookrightarrow B_{\infty,\infty}^0$}.

Consequently, we deduce that
  \begin{eqnarray*}
 \|u(t)\|_{LL}\leq C_0(1+\int_0^t\|u(\tau)\|_{LL}d\tau), 
\end{eqnarray*}
and by Gronwall's Lemma 
$$
\|u(t)\|_{LL}\leq C_0\exp(C_0t),\qquad\forall\, t\geq 0.
$$
This yields in particular
$$
  \|\omega(t)\|_{\lb}\leq C_0\exp{(C_0t)},\qquad\forall\, t\geq 0,
   $$
  as claimed.

\end{proof}

\subsection{ Existence} Let  \mbox{$\omega_0\in  L^p\cap \lb$} and $u_0=k\ast \omega_0, $
    with $K(x)=\frac{x^\perp}{2\pi|x|^2}.$ We take  \mbox{$\rho\in  \mathcal C^\infty_0$}, with  \mbox{$\rho\geq 0$}  and  \mbox{$\int\rho(x)dx=1$}  and set
$$
\omega_0^n=\rho_n\ast \omega_0,\qquad u_0^n= \rho_n\ast u_0,
$$
 where  \mbox{$\rho_n(x)=n^2\rho(nx)$}. Obviously,  \mbox{$\omega_0^n$}  is   a  \mbox{$C^\infty$}  bounded function for every \mbox{$n\in\mathbb N^*$}. Furthermore, thanks to  \eqref{eq:comp}, 
  $$
 \|\omega_0^n\|_{L^p}\leq \|\omega_0\|_{L^p}\qquad{\rm and}\qquad \|\omega_0^n\|_{\lb}\leq \|\omega_0\|_{\lb}.
  $$
The classical interpolation result between Lebesgue  and \mbox{${\rm BMO}$}  spaces (see \cite{GR} for more details) implies that 
$$
 \|\omega_0^n\|_{L^q}\leq \|\omega_0^n\|_{L^p\cap {\rm BMO}}\leq \|\omega_0\|_{L^p\cap {\rm BMO}} , \qquad \forall\, q\in[p,+\infty[.
$$
Since,  \mbox{$\omega_0^n\in L^p\cap L^\infty$}  then there exists a unique weak solution   \mbox{$u^n$}  with
 $$
\omega_n\in L^\infty(\R_+, L^p\cap L^\infty).
$$
according to the classical result of Yudovich \cite{Y1}.
According to Proposition \ref{apriori} one has
 \begin{eqnarray}
 \label{44}
 \|u^n(t)\|_{LL}+ \|\omega^n(t)\|_{L^p\cap\lb}\leq C_0\exp(C_0t),\qquad\forall\, t\in\R_+.
  \end{eqnarray}
  With this  uniform estimate  in hand, we can perform the same analysis as in the case  \mbox{$\omega_0\in L^p\cap L^\infty$}  (see paragraph 8.2.2 in \cite{Maj} for more explanation). For the convenience of the reader we briefly outline the  main arguments of the proof.
  
  If one denotes by  \mbox{$\psi_n(t,x)$}  the associated flow to  \mbox{$u^n$}  then 
  \begin{equation}
  \label{tt}
   \|\psi_n^{\pm1}(t)\|_{*}\leq C_0\exp(C_0t),\qquad\forall\, t\in\R_+.
  \end{equation}
  This yields the existence of explicit time continuous functions  \mbox{$\beta(t)>0$}  and  \mbox{$C(t)$}   such that
   $$
  |\psi_n^{\pm1}(t,x_2)-\psi_n^{\pm1}(t,x_1)|\leq C(t)|x_2-x_1|^{\beta(t)},\qquad \forall\, (x_1,x_2)\in\R^2\times\R^2.
   $$
  Moreover,
   $$
   |\psi_n^{\pm1}(t_2,x)-\psi_n^{\pm1}(t_1,x)|\leq |t_2-t_1|\|u^n\|_{L^\infty}\leq C_0|t_2-t_1|,\qquad \forall\, (t_1,t_2)\in\R_+\times\R_+.
    $$
Here, we have used the Biot-Savart law to get
$$
   \|u^n(t)\|_{L^\infty}\lesssim \|\omega^n(t)\|_{L^p\cap L^3}\leq\|\omega_0\|_{L^p\cap L^3}.
    $$
   The family  \mbox{$\{\psi_n,\, n\in\mathbb N\}$}  is bounded and equicontinuous on every compact  \mbox{$[0,T]\times \bar B(0,R)\subset \R_+\times\R^2$}. The Arzela-Ascoli
theorem implies
   the existence of a limiting particle trajectories  \mbox{$\psi(t,x)$}. Performing the same analysis for  \mbox{$\{\psi_n^{-1},\, n\in\mathbb N\}$}  we figure out  that  \mbox{$\psi(t,x)$}  is a Lebesgue  measure preserving  homeomorphism . Also, passing to the limit\footnote{ We take the pointwise limit in the definition formula and then take the supremum.} in \eqref{tt} leads to
    $$
    \|\psi_t\|_{*}=\|\psi^{-1}_t\|_{*}\leq C_0\exp(C_0t),\qquad \forall\, t\in\R_+.
    $$
   One defines,
    $$
   \omega(t,x)=\omega_0(\psi^{-1}_t(x)),\qquad u(t,x)=(k\ast_x \omega(t,.))(x).
    $$
We easily check that for every  \mbox{$q\in [p,+\infty[$}  one has
   \begin{eqnarray*}
  \omega^n(.,t)&\longrightarrow& \omega(.,t)\quad {\rm in }\,\, L^q.
  \\
   u^n(.,t)&\longrightarrow_x& u(.,t)\quad {\rm uniformly}. 
  \end{eqnarray*}
The last claim follows from the fact that
   $$
   \|u^n(t)-u(t)\|_{L^\infty}\lesssim \|\omega^n(t)-\omega(t)\|_{ L^p\cap L^3}.
   $$
 All this allows us  to pass to the limit in the integral equation on  \mbox{$\omega^n$}  and then to prove that  \mbox{$(u,\omega)$}  is  a weak solution to the vorticity-stream formulation of the 2D Euler system. Furthermore, the  convergence of 
   \mbox{$\{\omega^n(t)\}$}  in  \mbox{$L^1_{\text{loc}}$}  and \eqref{44} imply together that
   $$
  \|\omega(t)\|_{L^p\cap\lb}\leq C_0\exp(C_0t),\qquad \forall\,t\in\R_+.
   $$
  as claimed.
  
  The continuity of  \mbox{$\psi$}  and the preservation of Lebesgue  measure imply that  \mbox{$t\mapsto \omega(t)$}  is continuous\footnote{ By approximation we are reduced to the following situation:  \mbox{$g_n(x)\to g(x)$}  pointwise  and 
   $$\|g_n\|_{L^q}=\|g\|_{L^q}.
   $$
   This is enough to deduce that  \mbox{$g_n\to g$}  in  \mbox{$L^q$}  (see Theorem 1.9 in \cite{LL}  for instance). } with values in  \mbox{$L^q$}  for all  \mbox{$q\in [p,+\infty[$}. This implies in particular  that
   \mbox{$u\in \mathcal C([0,+\infty[, L^r(\er^d))$}  for every  \mbox{$r\in [\frac{2p}{2-p},+\infty]$}. 
 \subsection{ Uniqueness} Since the vorticity remains bounded in  \mbox{${\rm BMO}$}  space then the uniqueness of the solutions follows from Theorem 7.1 in \cite{Vishik1}. 
 Another way to prove that is to add the information    \mbox{$u\in \mathcal C([0,+\infty, L^\infty(\er^d))$}  (which is satisfied for the solution constructed above) to the theorem and in this case the uniqueness follows from Theorem 7.17 in \cite{bah-ch-dan}.
 
\

\end{document}